\documentclass[oneside,english]{article}

\usepackage{graphicx} 
\usepackage[T1]{fontenc}
\usepackage{amsmath}
\usepackage{amsthm}
\usepackage{amssymb}
\usepackage[all]{xy}
\usepackage{comment}

\usepackage{color}
\usepackage[nocompress]{cite}
\usepackage{tikz-cd}
\usepackage{mathtools}

\makeatletter
\newcommand{\nc}{\newcommand}
\theoremstyle{plain} 
\newtheorem{thm}{Theorem}
\nc{\bthm}{\begin{thm}} \nc{\ethm}{\end{thm}}
\newtheorem{prop}[thm]{Proposition}
\nc{\bprp}{\begin{prop}} \nc{\eprp}{\end{prop}}

\newtheorem{fact}[thm]{Fact}
\nc{\bfct}{\begin{fact}} \nc{\efct}{\end{fact}}
\newtheorem{prob}[thm]{Problem}
\nc{\bprb}{\begin{prob}} \nc{\eprb}{\end{prob}}
\newtheorem{lem}[thm]{Lemma}
\nc{\blem}{\begin{lem}} \nc{\elem}{\end{lem}}
\newtheorem{claim}[thm]{Claim}
\nc{\bclm}{\begin{claim}} \nc{\eclm}{\end{claim}}
\newtheorem{cor}[thm]{Corollary}
\nc{\bcor}{\begin{cor}} \nc{\ecor}{\end{cor}}
\newtheorem{conj}[thm]{Conjecture}
\nc{\bcnj}{\begin{conj}} \nc{\ecnj}{\end{conj}}
\theoremstyle{definition}
\newtheorem{defn}[thm]{Definition}
\nc{\bdfn}{\begin{defn}} \nc{\edfn}{\end{defn}}
\newtheorem{observation}[thm]{Observation}
\nc{\bobs}{\begin{observation}} \nc{\eobs}{\end{observation}}
\theoremstyle{remark}
\newtheorem{rem}[thm]{Remark}
\nc{\brem}{\begin{rem}} \nc{\erem}{\end{rem}}
\newtheorem{cnv}[thm]{Convention}
\nc{\bcnv}{\begin{cnv}} \nc{\ecnv}{\end{cnv}}
\newtheorem{exam}[thm]{Example}
\nc{\bexm}{\begin{exam}} \nc{\eexm}{\end{exam}}

\nc{\bpf}{\begin{proof}} \nc{\epf}{\end{proof}}
\nc{\be}{\begin{enumerate}}
	\nc{\ee}{\end{enumerate}}
\nc{\bi}{\begin{itemize}}
	\nc{\itm}{\item}
	\nc{\ei}{\end{itemize}}
\nc{\invlim}{\lim_{\leftarrow}}
\nc{\dirlim}{\lim_{\rightarrow}}
\nc{\mm}{\mathbf{m}}
\nc{\nn}{\mathbf{n}}
\nc{\kk}{\mathbf{k}}
\nc{\FF}{\mathcal{F}}
\nc{\CC}{\mathcal{C}}
\nc{\Span}{\operatorname{span}}
\nc{\Img}{\operatorname{Im}}
\nc{\rank}{\operatorname{rank}}
\nc{\proj}{\operatorname{proj}}
\nc{\F}{\mathbb{F}}
\nc{\Z}{\mathbb{Z}}
\nc{\Q}{\mathbb{Q}}
\nc{\Br}{\operatorname{Br}}
\newcommand\blfootnote[1]{
	\begingroup
	\renewcommand\thefootnote{}\footnote{#1}
	\addtocounter{footnote}{-1}
	\endgroup
}

\title{ Infinite free pro-$p$ products as absolute Galois groups}

\author{Tamar Bar-On}

\date{}
\begin{document}
	\maketitle
	\blfootnote{2020 Mathematics Subject Classification. Primary 20E18, Secondary 12F12.
		
		Key words and phrases. Absolute Galois groups. Free pro-$\mathcal{C}$ products.
		
		This work was partially supported by the Israel Science Foundation (grant No. 569/21) and  the BGU Center for Advanced Studies in Mathematics.}
	\begin{abstract}
		In this paper we prove a generalized version of Mel'nikov's Criterion, which has also been proved by Haran, and use it in order to connect the realization as absolute Galois groups of free profinite and free pro-$p$ products of a sheaf of pro-$p$ groups. 
	\end{abstract}
	\section*{Introduction}
	One of the main tasks in Galois theory these days is determining which profinite groups can be realized as absolute Galois groups. 
	
	In \cite{jarden1996infinite} Moshe Jarden suggested the following question: Let $n$ be a positive integer. Is the free profinite product of $n$ absolute Galois groups an absolute Galois group as well? This question has been answered in the affirmative in several papers, such as \cite{ershov1997free, ershov2001free} and \cite{haran2000free}. Moreover, it has been shown (see, for example, \cite[Theorem 3.4]{haran2000free}) that if $K_1,...,K_n$ all have characteristic $\ell$, then there is a field $K$ of characteristic $\ell$ such that $G_{K_1}\amalg\cdots \amalg G_{K_n}\cong G_K$. 
	
	In \cite[Remark 3.5]{haran2000free} the authors showed that if $G_1,...,G_n$ are pro-$p$ groups, then there exists an embedding  $G_1\amalg^p\cdots \amalg^p G_n\hookrightarrow  G_1\amalg\cdots \amalg G_n$, where $A\amalg^pB$ stands for the free pro-$p$ product of the pro-$p$ groups $A,B$. As an immediate consequence, they concluded that if $G_1,...,G_n$ are pro-$p$ groups that can be realized as absolute Galois groups, then so can their free pro-$p$ product. This result has been previously proved by Efrat and Haran \cite[Lemma 1.3]{efrat1994galois}.
	
	A natural follow-up question arises regarding general free pro-$\mathcal{C}$ products.
	
	Let $\mathcal{C}$ be an extension-closed variety of finite groups, i.e., a family of finite groups which is closed under taking quotients, subgroups and extensions. In \cite{Koenigsmann2005Products} Koenigsmann showed that in the context of absolute Galois groups, the only varieties one should consider for free pro-$\mathcal{C}$ products are the varieties of all finite groups and of finite $p$-groups. More precisely, he proved the following: Let $\mathcal{C}$ be a class of finite groups closed under forming subgroups, quotients, and extensions, and assume that the class of pro-$\mathcal{C}$ absolute Galois groups is closed under free pro-$\mathcal{C}$ products of a finite number of groups. Then $\mathcal{C}$ is either the class of all finite groups or, for some prime $p$, the class of all finite $p$-groups.

	A natural extension of Jarden's problem asks whether the free profinite product of \textit{infinitely} many absolute Galois groups can be realized as an absolute Galois group as well. When talking about free pro-$\mathcal{C}$ products of infinitely many profinite groups, we always mean free pro-$\mathcal{C}$ products \textit{over sheaves}, which are the natural generalization of the free pro-$\mathcal{C}$ product of finitely many groups, as defined below:

	Let $T$ be a profinite space and let $\mathcal{G}$ be a profinite space equipped with a continuous surjection $\pi:\mathcal{G}\to T$ such that for every $t\in T$, $\pi^{-1}(t)$ is a pro-$\mathcal{C}$ group, and the map $(x,y)\mapsto xy^{-1}$ is continuous with respect to the subspace topology on $\{(x,y):\pi(x)=\pi(y)\}$.  Such a triple $(\mathcal{G},T,\pi)$, which sometimes will be referred to as $\mathcal{G}$ when the other data are clear from the context, is called a \textbf{sheaf of pro-$\mathcal{C}$ groups}. We also use the notation $G_t$ for $\pi^{-1}(t)$. A morphism $\mathcal{G}\to A$ from a sheaf of pro-$\mathcal{C}$ groups to a pro-$\mathcal{C}$ group $A$ is a continuous map whose restriction to each $G_t$ is a homomorphism. The free pro-$\mathcal{C}$ product $G$ of a sheaf $\mathcal{G}$ is a pro-$\mathcal{C}$ group $G$, equipped with a morphism $\omega:\mathcal{G}\to G$ such that for every morphism $\beta:\mathcal{G}\to A$, where $A$ is a pro-$\mathcal{C}$ group, there exists a unique continuous homomorphism $\bar{\beta}:G\to A$ which satisfies $\bar{\beta}\circ\omega=\beta$. The free pro-$\mathcal{C}$ product of the sheaf $\mathcal{G}$ will be denoted by $\coprod^{\mathcal{C}}\mathcal{G}$, or, when the sheaf is clear from the context, by $\coprod_T^{\mathcal{C}}G_t$. If $\mathcal{C}$ is the variety of all finite groups, then we omit the $\mathcal{C}$ from the notation and simply write $\coprod\mathcal{G}$ ($\coprod_TG_t$). If $\mathcal{C}$ is the variety of all finite $p$-groups, we write $\coprod^p\mathcal{G}$. For an extensive discussion about free pro-$\mathcal{C}$ products of sheaves, as well as an internal point of view, the reader is referred to \cite[Chapter 5]{ribes2017profinite}.

	It turns out that the general version of Jarden's question to free profinite products of infinitely many absolute Galois groups has a negative answer. More precisely, in \cite{mel1999free} Mel'nikov gives an example of a sheaf consisting of profinite groups which can be realized as absolute Galois groups, whose free profinite product cannot be realized as an absolute Galois group. More counterexamples, as well as positive examples, appear in \cite{jarden2025infinite}. 
	
	In \cite{mel1999free}, Mel'nikov also gives the following criterion for the free profinite product of a sheaf of profinite groups to be an absolute Galois group.
	\begin{thm}[Mel'nikov's Criterion]\label{Mel'nikov's Criterion}
		Let $T$ be a separable profinite space and $\mathcal{G}=\bigcup_T G_t$ a sheaf of profinite groups defined over $T$, and let $\ell$ be $0$ or a prime number. Then $\coprod_TG_t$ is an absolute Galois group of a field of characteristic $\ell$ if and only if there exists a morphism $\varphi:\mathcal{G}\to G_F$ for some field $F$ of characteristic $\ell$, such that $G_F$ is a separable profinite group and for every $t\in T$ the restriction of $\varphi$ to $G_t$ is injective.
	\end{thm}
	Here and below $G_F$ stands for the absolute Galois group of $F$.
	
	The original requirement that $T$ and $G_F$ be separable was due to the fact that, at that time, Jarden’s question was only solved for free products of separable profinite groups. Since the class of absolute Galois groups is now known to be closed under free profinite products of finitely many groups, Mel'nikov's criterion can be generalized with no restrictions. This was briefly mentioned in \cite{bar2025restricted} with no further reference. Hence, we shall rephrase Theorem \ref{Mel'nikov's Criterion}:
	\begin{thm}[Generalized Mel'nikov's Criterion]\label{Generalized Mel'nikov's Criterion}
		Let $T$ be a  profinite space and $\mathcal{G}=\bigcup_T G_t$ a sheaf of profinite groups defined over $T$, and let $\ell$ be $0$ or a prime number. Then $\coprod_TG_t$ is an absolute Galois group of a field of characteristic $\ell$ if and only if there exists a morphism $\varphi:\mathcal{G}\to G_F$ for some field $F$ of characteristic $\ell$, such that for every $t\in T$ the restriction of $\varphi$ to $G_t$ is injective.
	\end{thm}

		\begin{proof}
			We prove the nontrivial implication; for the converse, note that if $\coprod_TG_t$ is an absolute Galois group of a field $K$ of characteristic $\ell$, then the canonical morphism $\omega:\mathcal{G}\to \coprod_TG_t$ is injective on each $G_t$ (\cite[Proposition 5.1.6 (c)]{ribes2017profinite}) and may serve as $\varphi$. First we recall the concept of a \textit{constant sheaf}. Let $A$ be a profinite group, and $T$ a profinite space. One easily checks that $(A\times T,T,\pi)$, where $\pi$ is the natural projection onto the second component, is a sheaf of profinite groups. We refer to sheaves that are built in this form as \textit{constant sheaves}.
			
			Now look at the map $\varphi':\mathcal{G}\to G_F\times T$ defined by $\varphi'(g)=(\varphi(g),\pi(g))$. Since the restriction of $\varphi$ to each $G_t$ is injective, $\varphi'$ is injective, and as $\mathcal{G}$ is compact it is a homeomorphism onto its image. Hence $\mathcal{G}$ can be identified with a sub-sheaf of $G_F\times T$. As was explained in the proof of \cite[Proposition 2.4]{mel1999free} (see also \cite[Theorem 5.5.6]{ribes2017profinite}), the closed subgroup of the free profinite product $\coprod_TG_F$ which is generated by the images of $G_t$, $t\in T$, is precisely the free profinite product $\coprod_TG_t$. Hence $\coprod_TG_t$ can be identified with a closed subgroup of $\coprod_TG_F$. Recall also that a closed subgroup of $G_K$ is the absolute Galois group of an algebraic extension of $K$, and in particular of a field of the same characteristic. We conclude that it is enough to prove that the free profinite product of a constant sheaf which consists of an absolute Galois group of a field of characteristic $\ell$ can be realized as an absolute Galois group of a field of characteristic $\ell$ as well.
			
			Let $G$ be an absolute Galois group of a field of characteristic $\ell$ and $T$ a profinite space. Consider $F[T]$, the free profinite group over the profinite space $T$ (for the definition see, for example, \cite[Subsection 3.3]{ribes2000profinite}). Recall that $F[T]$ can be realized as an absolute Galois group of a field of any characteristic. Indeed, free profinite groups are trivially projective, and all projective profinite groups can be realized as absolute Galois groups of fields of any characteristic (see, for example, \cite[Proposition 1.3]{mel1999free} and \cite{fried2005field}).
			
			By the affirmative solution to Jarden's question, in the form of \cite[Theorem 3.4]{haran2000free}, $G\coprod F[T]$ can be realized as an absolute Galois group of a field of characteristic $\ell$. By \cite[Lemma 2.3]{mel1999free} the normal closure of $G$ in $G\coprod F[T]$ is isomorphic to $\coprod_{h\in F[T]}G_h$ where $G_h=hGh^{-1}$ for every $h\in F[T]$. In particular $G_h\cong G$. Since $T$ is a closed subset of $F[T]$,  $\coprod_{t\in T}tGt^{-1}\leq \coprod_{h\in F[T]}hGh^{-1}$ (\cite[Lemma 5.5.1]{ribes2017profinite}). Observe that the map from $G\times T$ to $\bigcup_{t\in T}(tGt^{-1}\times \{t\})\leq (G\coprod F[T])\times T$ which is defined by $(g,t)\mapsto (tgt^{-1},t)$ is a continuous bijection between compact Hausdorff spaces, hence a homeomorphism, and thus an isomorphism of sheaves. Hence it induces an isomorphism of the free profinite product $\coprod_TG\cong \coprod_TtGt^{-1}$. Thus, $ \coprod_{T}G$ can be realized as an absolute Galois group of a field of characteristic $\ell$.
		\end{proof}

	The above general equivalent condition has been recently proved also in \cite{haran2025free}. However, in this paper we shall refer to it as \textbf{Mel'nikov Criterion}.
	
	As stated above, the only free pro-$\mathcal{C}$ products that are interesting in Galois theory are the free profinite and free pro-$p$ products. While for free profinite products over infinite sheaves we already have a criterion, the realization of free pro-$p$ products over infinite sheaves has hardly been studied so far (see, however, \cite{bar2025restricted} for the restricted case). In this paper we use the general version of Mel'nikov's Criterion in order to connect the realization of free profinite and free pro-$p$ products of infinite sheaves of pro-$p$ groups. More precisely, the main theorem of the paper is the following:
	\begin{thm}[Main Theorem]
		Let $\mathcal{G}=\bigcup_TG_t$ be a sheaf of pro-$p$ groups. Then $\coprod_TG_t$ is an absolute Galois group of a field of characteristic $\ell$ if and only if $\coprod_T^pG_t$ is an absolute Galois group of a field of characteristic $\ell$.
	\end{thm}

	\section{Main Theorem}

Our main tool toward a proof of the main theorem is a natural generalization of \cite[Remark 3.5]{haran2000free}. Before proving it we need the following lemma. In it, for a closed subgroup $H$ of a profinite group $G$ and $a\in G$, we write $H^a=a^{-1}Ha$, and $N\trianglelefteq_o G$ (resp. $K\trianglelefteq_c G$) means that $N$ is an open (resp. closed) normal subgroup of $G$. Moreover, for $K\trianglelefteq_c G$ and $u\in G/K$, the condition $H^u\leq PK$ means that $H^a\leq PK$ for some, or equivalently every, $a\in G$ with $aK=u$.
	\begin{lem}\label{lemma continuous conjugation}
		Let $G$ be a profinite group, let $P$ be a $p$-Sylow subgroup of $G$, let $T$
		be a profinite space, and let $\{H_t\}_{t\in T}$ be a family of closed pro-$p$
		subgroups of $G$ such that the set
		\[
		\Gamma=\{(t,g)\in T\times G : g\in H_t\}
		\]
		is closed in $T\times G$. Then there exists a continuous map
		$\sigma\colon T\to G$ such that $H_t^{\sigma(t)}\le P$ for every $t\in T$.
	\end{lem}

\begin{proof}
	\textit{Step 1:} 	Let $N\trianglelefteq_o G$ and let $X\subseteq G$ be a union of cosets of $N$.
	Then the set $\{t\in T : H_t\subseteq X\}$ is open in $T$.
	
	\textit{Proof:} 	Since $G\setminus X$ is a union of cosets of the open subgroup $N$, it is
	closed in $G$. Hence $\Gamma\cap\bigl(T\times(G\setminus X)\bigr)$ is a closed
	subset of the compact space $T\times G$, and so it is compact. The complement
	of $\{t\in T : H_t\subseteq X\}$ in $T$ is the image of this compact set under
	the projection $T\times G\to T$. Therefore it is compact, and hence closed
	in $T$.
	
	\textit{Step 2:} 	Let $K,N\trianglelefteq_c G$. Then $PK\cap PN=P(K\cap N)$.
	
	\textit{Proof:} 	Put $K'=K\cap N$. Clearly $PK'\le PK\cap PN$. The natural homomorphism
	$G/K'\to G/K\times G/N$ is injective, and it maps $(PK\cap PN)/K'$ into
	$PK/K\times PN/N$, which is a pro-$p$ group. Hence $(PK\cap PN)/K'$ is a
	pro-$p$ subgroup of $G/K'$ containing $PK'/K'$. Since $PK'/K'$ is a $p$-Sylow
	subgroup of $G/K'$, it is a maximal pro-$p$ subgroup of $G/K'$. Thus
	$(PK\cap PN)/K'=PK'/K'$, that is, $PK\cap PN=PK'$. 
	
	\textit{Step 3:} 	Let $K\trianglelefteq_c G$ and $N\trianglelefteq_o G$, and put $K'=K\cap N$.
	Let $\sigma: T\to G/K$ be a continuous map such that
	$H_t^{\sigma(t)}\leq PK$ for every $t\in T$. Then there exists a continuous map
	$\sigma': T\to G/K'$  such that $\sigma=\pi_{K',K}\circ \sigma'$, and
	$H_t^{\sigma'(t)}\le PK'$ for every $t\in T$. Here, and in what follows, for closed normal subgroups $K_1\leq K_2$ of $G$ we denote by $\pi_{K_1,K_2}$ the natural projection $G/K_1\to G/K_2$; in particular $\pi_{K',K}\colon G/K'\to G/K$.
	
	\textit{Proof:} 	Recall that $\pi_{K,KN}\colon G/K\to G/KN$ and $\pi_{N,KN}\colon G/N\to G/KN$ denote the natural
	projections. Note that $G/KN$ is finite.

	First we claim that for every $t\in T$ there exists $a\in G$ with $aK=\sigma(t)$ and
	$H_t^a\leq P$. Indeed, choose $a_0\in G$ with $a_0K=\sigma(t)$. Then
	$H_t^{a_0}$ is a pro-$p$ subgroup of $PK$, and $P$ is a $p$-Sylow subgroup
	of $PK$, being a maximal pro-$p$ subgroup of $G$. By Sylow theory in $PK$
	there is $c\in PK$ with $H_t^{a_0c}\le P$. Since $PK=KP$, we can write
	$c=kq$ with $k\in K$ and $q\in P$. Then
	$H_t^{a_0k}=\bigl(H_t^{a_0c}\bigr)^{q^{-1}}\le P^{q^{-1}}=P$, so $a=a_0k$
	is as required.

	\textit{Step 4:} For $x\in G/N$ choose $b\in G$ with $x=bN$, and set
	\[
	U_x=\{t\in T : \pi_{K,KN}(\sigma(t))=\pi_{N,KN}(x)\}\ \cap\
	\{t\in T : H_t^b\le PN\}.
	\]
	The second condition does not depend on the choice of $b$, since $N\le PN$.
	The first set is clopen in $T$, because $\pi_{K,KN}\circ\sigma$ is continuous and
	$G/KN$ is finite. The second set equals $\{t\in T : H_t\subseteq bPNb^{-1}\}$,
	and $bPNb^{-1}=bPb^{-1}N$ is a union of cosets of $N$; hence it is open by
	Step 1. Thus each $U_x$ is open in $T$. Moreover, if $t\in T$
	and $a$ is as in Step 3, then $\pi_{K,KN}(\sigma(t))=aKN=\pi_{N,KN}(aN)$ and
	$H_t^a\leq P\leq PN$, so $t\in U_{aN}$. Hence $\{U_x\}_{x\in G/N}$ is a finite
	open cover of $T$.

	Since $T$ is a profinite space, there exist a partition
	$T=C_1\cup\dots\cup C_m$ of $T$ into pairwise disjoint clopen sets and
	elements $x_1,\dots,x_m\in G/N$ such that $C_i\subseteq U_{x_i}$ for every $i$.
	Define $\tau: T\to G/N$ by $\tau(C_i)=\{x_i\}$. Then $\tau$ is
	continuous and $t\in U_{\tau(t)}$ for every $t\in T$. Consider the natural map
	\[
	\theta\colon G/K'\longrightarrow G/K\times_{G/KN}G/N
	=\{(u,v)\in G/K\times G/N : \pi_{K,KN}(u)=\pi_{N,KN}(v)\}.
	\]
	It is injective because $K'=K\cap N$. It is surjective: if
	$\pi_{K,KN}(aK)=\pi_{N,KN}(bN)$, then $b=akn$ for some $k\in K$ and $n\in N$, and
	$\theta(akK')=(aK,bN)$. Being a continuous bijection between compact Hausdorff
	spaces, $\theta$ is a homeomorphism. Since $\pi_{K,KN}(\sigma(t))=\pi_{N,KN}(\tau(t))$
	for every $t\in T$, the map
	$\sigma'=\theta^{-1}\circ(\sigma,\tau)\colon T\to G/K'$ is well defined and
	continuous, and it lifts $\sigma$.
	
	Finally, let $t\in T$ and let $b\in G$ with $bK'=\sigma'(t)$. Then
	$bK=\sigma(t)$, so $H_t^b\le PK$. Also $bN=\tau(t)$ and
	$t\in U_{\tau(t)}$, so $H_t^b\le PN$. By Step 2
	$H_t^b\le PK\cap PN=PK'$.
	
	\textit{Final step:} We use Zorn's Lemma to prove the lemma.
	
		Let $\Omega$ be the set of all pairs $(K,\sigma_K)$, where
	$K\trianglelefteq_c G$ and $\sigma_K\colon T\to G/K$ is a continuous map
	such that $H_t^{\sigma_K(t)}\le PK$ for every $t\in T$. Order $\Omega$ by
	setting $(K_1,\sigma_1)\preceq(K_2,\sigma_2)$ if $K_2\le K_1$ and $\sigma_1=\pi_{K_2,K_1}\circ \sigma_2$, where, as above, $\pi_{K_2,K_1}\colon G/K_2\to G/K_1$ is the natural projection.
 The pair consisting of $G$ and the constant map
	$T\to G/G$ belongs to $\Omega$, so $\Omega\neq\emptyset$.
	
	Let $\{(K_i,\sigma_i)\}_{i\in I}$ be a nonempty chain in $\Omega$ and put
	$K=\bigcap_{i\in I}K_i$. Since the subgroups $K_i$ form a chain,
	$G/K=\varprojlim_{i\in I}G/K_i$, and the compatible maps $\sigma_i$ define a
	continuous map $\sigma\colon T\to G/K$. Let $t\in T$ and $a\in G$ with
	$aK=\sigma(t)$. Then $aK_i=\sigma_i(t)$, and hence $H_t^a\le PK_i$, for every
	$i\in I$. Observe that $\bigcap_{i\in I}PK_i=PK$ (see \cite[Proposition 2.1.4]{ribes2000profinite}).
Consequently $H_t^a\le PK$, and $(K,\sigma)\in\Omega$ is an
	upper bound for the chain.
	
	By Zorn's lemma, $\Omega$ has a maximal element $(K,\sigma_K)$. Suppose that
	$K\neq 1$. Since the open normal subgroups of $G$ intersect trivially, there
	is $N\trianglelefteq_o G$ with $K\not\le N$. By Step 3
	there is a continuous map $\sigma'\colon T\to G/(K\cap N)$ such that
	$(K\cap N,\sigma')\in\Omega$ and $(K,\sigma_K)\preceq(K\cap N,\sigma')$. As
	$K\cap N\neq K$, this contradicts the maximality of $(K,\sigma_K)$. Therefore
	$K=1$, and $\sigma=\sigma_K\colon T\to G$ is as required.
\end{proof}	
 \begin{cor}\label{existenes of alpha}
 	Let $(\mathcal{G},T,\pi)$ be a sheaf of pro-$p$ groups and $G$ be its free profinite product. Let $P$ be a $p$-Sylow subgroup of $G$. Then there exists a continuous map $\alpha:T\to G$ such that $H_t^{\alpha(t)}\leq P$ for every $t\in T$, where $H_t$ denotes the image of $G_t$ in $G$ under the canonical morphism $\omega:\mathcal{G}\to G$. 
 \end{cor}
\begin{proof}
	By Lemma \ref{lemma continuous conjugation}, the only thing we need to show is that $\Gamma=\{(t,g):g\in H_t\}$ is closed in $T\times G$. Indeed, $H_t=\omega(G_t)$ is a closed pro-$p$ subgroup of $G$, and $\Gamma$ is the image of the compact space $\mathcal{G}$ under the continuous map $(\pi,\omega):\mathcal{G}\to T\times G$, hence it is closed.
\end{proof}

	Now we can prove the required generalization of \cite[Remark 3.5]{haran2000free}.
	\begin{prop}\label{pro-p versus profinite}
		Let $\mathcal{G}$ be a sheaf of pro-$p$ groups. Then there is an embedding of $\coprod_T^pG_t$ into $\coprod_TG_t$.
	\end{prop}
	\begin{proof}
	Let $G=\coprod_TG_t$, let $P$ be a $p$-Sylow subgroup of $G$, and let $\alpha:T\to G$ be as in Corollary \ref{existenes of alpha}. Recall that by \cite[Theorem 5.6.1]{ribes2017profinite} $\coprod_T^pG_t$ is the maximal pro-$p$ quotient of $\coprod_TG_t$. In particular there exists an epimorphism $\beta:\coprod_TG_t\to \coprod_T^pG_t$, which satisfies $\beta\circ\omega=\omega_p$, where $\omega:\mathcal{G}\to \coprod_TG_t$ and $\omega_p:\mathcal{G}\to \coprod_T^pG_t$ are the canonical morphisms. Since $\beta$ is an epimorphism, $\beta(P)$ is a $p$-Sylow subgroup of $\coprod_T^pG_t$. However, $\coprod_T^pG_t$ is a pro-$p$ group, thus the restriction of $\beta$ to $P$, which we again denote by $\beta$, is an epimorphism as well. By \cite[Proposition 2.2.2]{ribes2000profinite} there exists a continuous set-theoretic section $\sigma:\coprod_T^pG_t\to P$ of $\beta$. For $t\in T$ put $c(t)=\alpha(t)\sigma(\beta(\alpha(t)))^{-1}$, so that $c:T\to G$ is continuous and $\beta(c(t))=1$. Eventually, define $\gamma:\mathcal{G}\to P$ by $\gamma(x)=\omega(x)^{c(\pi(x))}$; note that $\gamma$ indeed takes its values in $P$, since $\omega(x)^{\alpha(t)}\in P$ for every $x\in G_t$ and $\sigma(\beta(\alpha(t)))\in P$. This is a continuous map which is a homomorphism on each $G_t$ and hence induces a homomorphism $\gamma':\coprod_T^pG_t\to P$ such that $\gamma'\circ\omega_p=\gamma$. Eventually, for all $t\in T$ and $x\in G_t$, $\beta(\gamma(x))=\omega_p(x)^{\beta(c(t))}=\omega_p(x)$. One concludes by the uniqueness property of free pro-$p$ products that $\beta\circ\gamma'=id_{\coprod_T^pG_t}$ and hence $\gamma'$ is injective.
	\end{proof}
	Now we can easily deduce the main theorem.
	\begin{thm}
		Let $\mathcal{G}=\bigcup_TG_t$ be a sheaf of pro-$p$ groups. Then $\coprod_TG_t$ is an absolute Galois group of a field of characteristic $\ell$ if and only if $\coprod_T^pG_t$ is an absolute Galois group of a field of characteristic $\ell$.
	\end{thm}
	\begin{proof}
		Assume first that $\coprod_T^pG_t$ is an absolute Galois group of a field of characteristic $\ell$. Since the natural map $\mathcal{G}\to \coprod_T^pG_t$ is injective on each fiber $G_t$ (\cite[Proposition 5.1.6 (c)]{ribes2017profinite}), Theorem \ref{Generalized Mel'nikov's Criterion} implies that so is $\coprod_TG_t$. Conversely, assume that $\coprod_TG_t$ is an absolute Galois group of a field $K$ of characteristic $\ell$. By Proposition \ref{pro-p versus profinite}, $\coprod_T^pG_t$ is isomorphic to a closed subgroup of $\coprod_TG_t$, and hence it is the absolute Galois group of an algebraic extension of $K$, which is again of characteristic $\ell$.
	\end{proof}
	
As we are discussing the connection between free profinite and free pro-$p$ products of sheaves of pro-$p$ groups, we take this opportunity to prove another connection, between free pro-$\mathcal{C}$ and free pro-$\mathcal{C}'$ products of the same sheaf. For that
we give a slight modification of \cite[Theorem 5.3.4]{ribes2017profinite}. 
\begin{lem}\label{preserving the variety}
	Let $\mathcal{C}\subseteq \mathcal{C}'$ be two extension-closed varieties of finite groups. Let $\mathcal{G}$ be a sheaf of pro-$\mathcal{C}$ groups and let $G=\coprod^{\mathcal{C}'}_TG_t$ be its free pro-$\mathcal{C}'$ product. Then there exists a directed partially ordered set $I$, and two inverse systems over $I$, one $(B_i,\psi_{ij},I)$ of pro-$\mathcal{C}'$ groups and the other $(T_i,\rho_{ij},I)$ of finite discrete spaces such that:
	\begin{enumerate}
		\item $T=\varprojlim_IT_i$.
		\item For each $i\in I$, $B_i=\coprod_{\tau\in T_i}^{\mathcal{C}'}B_{i,\tau}$ where each $B_{i,\tau}$ is a group in $\mathcal{C}$.
		\item $\psi_{ij}(B_{i,\tau})\leq B_{j,\tau'}$ whenever $i\geq j$ and $\rho_{ij}(\tau)=\tau'$.
		\item For every $t=(\tau_i)\in \varprojlim_IT_i=T$, $G_t=\varprojlim_{i\in I}B_{i,\tau_i}$. 
		\item $G=\varprojlim_{i\in I}B_i$.
	\end{enumerate}
\end{lem}
\begin{proof}
	Let $H=\coprod_T^{\mathcal{C}}G_t$ be the free pro-$\mathcal{C}$ product of $\mathcal{G}$. Use \cite[Theorem 5.3.4]{ribes2017profinite} to express $H$ as an inverse limit of free pro-$\mathcal{C}$ products of finitely many groups in $\mathcal{C}$ as below:
	There exists a directed partially ordered set $I$, and two inverse systems over $I$, one $(A_i,\varphi_{ij},I)$ of pro-$\mathcal{C}$ groups and the other $(T_i,\rho_{ij},I)$ of finite discrete spaces such that:
	\begin{enumerate}
		\item $T=\varprojlim_IT_i$.
		\item For each $i\in I$, $A_i=\coprod_{\tau\in T_i}^{\mathcal{C}}A_{i,\tau}$ where each $A_{i,\tau}$ is a group in $\mathcal{C}$.
		\item $\varphi_{ij}(A_{i,\tau})\leq A_{j,\tau'}$ whenever $i\geq j$ and $\rho_{ij}(\tau)=\tau'$.
		\item For every $t=(\tau_i)\in \varprojlim_IT_i=T$, $G_t=\varprojlim_{i\in I}A_{i,\tau_i}$. 
		\item $H=\varprojlim_{i\in I}A_i$.
	\end{enumerate}

	Now define $B_{i,\tau}=A_{i,\tau}$ for every $i\in I,\tau\in T_i$ and $B_i= \coprod_{\tau\in T_i}^{\mathcal{C}'}A_{i,\tau}$. We need to define maps $\psi_{ij}:B_i\to B_j$ for all $i\geq j$ in $I$. We do it as follows: Let $\tau\in T_i$ and $\tau'=\rho_{ij}(\tau)$. By assumption there are homomorphisms $A_{i,\tau}\to A_{j,\tau'}$ and hence there are homomorphisms $A_{i,\tau}\to B_j$. Since $B_j$ is a pro-$\mathcal{C}'$ group, by the universal property of the free pro-$\mathcal{C}'$ product $B_i$ these homomorphisms can be extended uniquely to a homomorphism $\psi_{ij}:B_i\to B_j$. From the uniqueness one deduces that the homomorphisms $\psi_{ij}$ are compatible. Hence we have an inverse system over $I$, $(B_i,\psi_{ij},I)$. By \cite[Proposition 5.3.2]{ribes2017profinite}, $\varprojlim_{i\in I}B_i\cong \coprod^{\mathcal{C}'}_TG_t$ and we are done.
\end{proof}

	\section*{Acknowledgements}
	The proof of Lemma~\ref{lemma continuous conjugation} was carried out by Claude, an AI assistant developed by Anthropic, in response to a question of the author (September 2026). Claude also assisted in checking the correctness of the arguments and in rewriting parts of the paper for readability.
	
	\bibliographystyle{plain}

\end{document}